\documentclass[10pt]{article}
\usepackage{cite,amscd,mathrsfs,amsmath,amsthm}
\usepackage{mathrsfs}
\usepackage{amssymb}
\usepackage{caption}
\usepackage[left=2.54cm,right=2.54cm,top=2.54cm,bottom=2.54cm]{geometry}
\usepackage{listings,appendix,enumitem,tikz,multirow}
\newtheorem{theorem}{Theorem}[section]

\newtheorem{defn}{Definition}[section]
\newtheorem{lem}{Lemma}[section]
\newtheorem{lemma}[lem]{Lemma}

\newtheorem{observe}{Observation}[section]
\newtheorem{prop}{Proposition}[section]
\newtheorem{remark}{Remark}[section]

\newtheorem{cor}{Corollary}[section]

\newcommand{\br}{\ensuremath{\mathbf{r}}}
\newcommand{\bc}{\ensuremath{\mathbf{c}}}
\newcommand{\cR}{\ensuremath{\mathcal{R}}}
\newcommand{\cS}{\ensuremath{\mathcal{S}}}
\everymath{\displaystyle}
\begin{document}

\title{\bf On Combinatorial Rectangles with Minimum $\infty-$Discrepancy}

\author{Chunwei Song\footnote{School of Mathematical Sciences \& LMAM, Peking University, Beijing 100871, P.R. China {\tt csong@math.pku.edu.cn},
{\tt byao@pku.edu.cn}}, Bowen Yao\footnotemark[2] \thanks{The authors were partially supported by NSF of China grant \#11771246.}}

\date{}

\maketitle

\begin{abstract}
A combinatorial rectangle may be viewed as a matrix whose entries are all $\pm1$.
The discrepancy of an $m \times n$ matrix is the maximum among the absolute values of its $m$
row sums and $n$ column sums.
In this paper, we investigate combinatorial rectangles with minimum discrepancy
($0$ or $1$ for each line depending on the parity).
Specifically, we get explicit formula for the number of matrices with minimum $L^\infty$-discrepancy
up to 4 rows, and establish the order of magnitude of the number of such matrices
with $m$ rows and $n$ columns while $m$ is fixed and $n$ approaches infinity.
By considering the number of column-good matrices with a
fixed row-sum vector, we have developed a theory of
decreasing criterion on based row-sum vectors with majorization relation,
which turns out to be a helpful tool in the proof of our main theorems.
\end{abstract}

\noindent \vskip3mm \noindent 2010 {\it MSC}:
Primary: 05A05,05A16; Secondary: 05A15,11K38


\section{Introduction}

Considered with respect to various kinds of mathematical objects, such as functions, groups, natural numbers and other summable or integrable  mathematical objects,
discrepancy is widely interested in the fields of analytic number theory, numerical analysis, P.D.E., graph theory and so on \cite{LSV86}.

In discrete mathematics, especially in combinatorics, discrepancy is taken over a finite set $\Omega$.
The $L^p$-discrepancy of a family of subsets $\mathscr{F}\subset2^\Omega$
with respect to a coloring $\chi:\Omega\rightarrow\{1,-1\}$ is defined to be \cite{BJS95},
$$\textrm{disc}_p(\mathscr{F},\chi)=\left(\frac1{|\mathscr{F}|}\sum_{X\in\mathscr{F}}\left|\chi(X)\right|^p\right)^{\frac1p}
=\left(\frac1{|\mathscr{F}|}\sum_{X\in\mathscr{F}}\left|\sum_{x\in X}\chi(x)\right|^p\right)^{\frac1p}.$$


A combinatorial rectangle is a bi-colored $m$ by $n$ grid with each square colored either red or blue \cite{SUN03}.
Discrepancy on combinatorial rectangles is studied in, say, \cite{ADLS02}. In the language above,
the family of subsets $\mathscr{F}$ is the set of all lines, i.e., rows and columns, and each subset $F$ is a row or a column.
A coloring of the grids $\chi:\Omega\rightarrow\{1,-1\}$ corresponds to filling the grids with $\pm1$.
Clearly, a combinatorial rectangle may simply be viewed a $\pm 1$ matrix.
The enumeration of nonnegative integer matrices has been a popular topic in combinatorics,
with results and useful techniques of enumerations of sparse or dense matrices given
by \cite{Hen95}, \cite{BBK72}, \cite{GC77}, \cite{AH09}, etc.

In this paper, we investigate combinatorial rectangles with minimum discrepancy ($0$ or $1$ for each line depending on the parity).
This is of particular interests because of the attainment of optimum balance (see \cite{CN05} or \cite{GD04}).

Specifically, the $L^\infty$-discrepancy of a $\pm1$ matrix $M$, which uniquely corresponds to a combinatorial rectangle, is as follows.
\begin{defn}
The $L^\infty$-discrepancy $\delta(M)$ of a $\pm1$ matrix $M=\{a_{ij}\}_{1\leq i\leq m}^{1\leq j\leq n}$ is defined as
$$\delta(M)= \lim_{p\rightarrow\infty}\left(\frac{1}{m+n} \left( \sum_{j=1}^n |\sum\limits_{i=1}^m a_{ij}|^p+ \sum_{i=1}^n |\sum\limits_{j=1}^n a_{ij}|^p\right)\right)^{\frac1p}
=\max\{\max_{1\leq j\leq n}|\sum_{i=1}^m a_{ij}|,\max_{1\leq i\leq m}|\sum_{j=1}^n a_{ij}|\},$$
i.e., $\delta(M)$ is the maximum value among all absolute values of row-sums and column-sums of
$M$.
\end{defn}

Let $A(m,n)$ denote the collection of all $\pm1$ matrices with $m$ rows and $n$ columns such that the discrepancy (i.e. the absolute value of the sum of
each row and column) is no more than $1$.
We say that a $\pm1$ matrix $M$ is ``good" if $M$ is in $A(m,n)$.

By checking at certain alternating arrangement pattern,
the minimum discrepancy is easily obtained.
Clearly, for any good matrix,
if the number of rows $m$ is even, the sum of $\pm1$'s in every column must be 0 with the number of $+1$'s and $-1$'s both $\frac n2$.
While if $m$ is odd,
the sum of every column is $1$ or $-1$ with the number of  $+1$'s and $-1$'s $\frac {m\pm1}2$ and $\frac {m\mp1}2$, respectively.
Let $\alpha(m,n)=|A(m,n)|$. It is an intriguing question to
estimate $\alpha(m,n)$, the total number of good matrices.

The case of $m=3$ is investigated in \cite{Hen95}.

\subsection{Preliminaries and Main results}

First to warm up we observe the following fact which is straightforward
while being helpful to understand the mechanism.

\begin{observe}\label{observe:1}
In case of one-row or two-row matrices, clearly,
$$\alpha(1,n)=\alpha(2,n)=\binom{2\lfloor\frac{n+1}{2}\rfloor}{\lfloor\frac{n+1}{2}\rfloor}.$$
\end{observe}

\begin{proof}
Note that $\alpha(1,n) = {n \choose \lfloor \frac{n}{2} \rfloor} \cdot 2^{\chi(2 \nmid n)}$,
where by convention $\chi(P)$ is to be taken equal to one if the statement $P$ is true and zero
if $P$ is false. Moreover, trivially $\alpha(2,n)=\alpha(1,n)$.
\end{proof}


The following results are useful in that they help
reduce discussions into neighboring cases.

\begin{observe}\label{0}
If $2|m$, then $$\alpha(m,n) \leq \alpha(m-1,n).$$
Moreover, if $2|m,2|n$, then $$\alpha(m,n)=\alpha(m-1,n)=\alpha(m,n-1)\leq\alpha(m-1,n-1).$$
\end{observe}

\begin{proof}
Suppose $2|m$. Note that deletion of the uppermost row from any arbitrary matrix in $A(m,n)$ results in a
uniquely decided matrix in $A(m-1,n)$. This may not be revertible, though, as completing every column of
a matrix in $A(m-1,n)$ with opposite-signed sum of the column entries might generate an extra row of
unwanted row sum. Thus $\alpha(m,n)\leq\alpha(m-1,n)$ for even $m$.

In case $m$ and $n$ are both even, the completed $m$th row will not have row sum problem. In fact,
for any matrix $M=\{a_{ij}\}_{1\leq i\leq m-1}^{1\leq j\leq n}\in A(m-1,n)$,
an extra row ($a_{mj}=-\sum_{i=1}^{m-1}a_{ij}, 1 \leq j \leq n$) is added to the top of $M$ according
to the above map. Because $n$ is even, every row sum of $M$ must be zero, and hence,
$$\begin{aligned}
\sum_{j=1}^na_{mj}&=\sum_{j=1}^n (-\sum_{i=1}^{m-1}a_{ij}) = -\sum_{i=1}^{m-1} \sum_{j=1}^n a_{ij}=-\sum_{i=1}^{m-1}0=0.
\end{aligned}$$
So the enhanced matrix $M'=\{a_{ij}\}_{1\leq i\leq m}^{1\leq j\leq n}$ belongs to $A(m,n)$ and
the map described above is a bijection. Thus when $m$ and $n$ are both even
we have $\alpha(m,n)=\alpha(m-1,n)$, and by symmetry $\alpha(m,n)=\alpha(m,n-1)$ in this case.


\end{proof}

\begin{prop}\label{prop:geq}
If $2 \nmid m$, i.e., $m-1$ is even, then $$\alpha(m,n)\geq\alpha(1,n)\alpha(m-1,n).$$
\end{prop}

\begin{proof}
Note that when $m-1$ is even, every column of an arbitrarily chosen matrix $M$ in $A(m-1,n)$ has discrepancy 0.
Therefore any row matrix
from $A(1,n)$ adding to the top of $M$ will result in a uniquely decided matrix in $A(m,n)$.
\end{proof}

In this paper we have done the followings. For general $n$, we evaluate the exact values of $\alpha(m,n)$
for $m$ up to 4, and analyze the structure of $A(3,n)$ by mapping the
matrices bijectively to certain 3-dimensional walks.
For fixed $m$, we determine the limit of the ratio of adjacent $\alpha(m,n)$ pairs.
Finally, we establish the magnitude of $\alpha(m,n)$ in general.

The next few results are proved in later sections.

\begin{theorem}\label{1}If $n=2k$ is even, then
$$\alpha(3,2k)=\sum_{i=0}^{k}\binom{n}{k-i,k-i,i,i}\binom{2i}{i}=\binom{n}{k} \sum_{i=0}^{k}\binom{k}{i}^2\binom{2i}i.$$
If $n=2k+1$ is odd, then
$$\alpha(3,2k+1)=2\binom{n}{k}\sum_{i=0}^{k}\binom{k}{i}\left[\binom{2i+1}{i+1}\binom{k+1}i+2\binom{k+1}{i+1}\binom{2i+1}{i}\right].$$
\end{theorem}

\begin{theorem}\label{n=4}
$$\alpha(4,n)=\binom{2\lceil \frac{n}{2} \rceil}{\lceil \frac{n}{2} \rceil} \sum_{i=0}^{\lceil \frac{n}{2} \rceil}\binom{\lceil \frac{n}{2} \rceil}{i}^2\binom{2i}i.$$
\end{theorem}

\begin{theorem}\label{thm:bij3}
The exists a bijection between $A_{3, n}$ and the collection of $n$-step walks that start from the origin and end at either the origin (when
$n$ is even), or one of $\{(1,0,0),(-1,0,0),(0,1,0),(0,-1,0),(0,0,1),(0,0,-1),\\
(1,1,1), (-1,-1,-1)\}$ (when $n$ is odd).
\end{theorem}



\begin{theorem}\label{porp}For fixed integer $m$,
$$\lim_{n\rightarrow\infty}\frac{\alpha(m,n+2)}{\alpha(m,n)}
=\alpha(m,1)^2
={\binom{2\lfloor\frac{m+1}{2}\rfloor}{\lfloor\frac{m+1}{2}\rfloor}}^2.$$
\end{theorem}

\begin{theorem}\label{appro}For fixed integer $m$,
$$\alpha(m,n)=\Theta\left(\frac{\alpha(m,1)^n}{n^{\lfloor\frac{m-1}{2}\rfloor+\frac12}}\right),$$
while $n$ tends to infinity, here $\alpha(2m-1,1)=\alpha(2m,1)=\binom{2m}m$.
\end{theorem}

The paper is organized as follows.


In Section \ref{34}, we prove Theorems \ref{1} and \ref{thm:bij3}. The exact values of $\alpha(3,n)$ and $\alpha(4,n)$
are established via discussing the first two rows of the matrix.
We also construct a bijection between $A_{3, n}$ and the collections of $n$-step walks in 3-dimensional space.

In Section \ref{LB}, we prove Theorems \ref{porp} and \ref{appro}.
In that process,  we develop a theory of
decreasing criterion on based row-sum vectors with majorization relation
(most notably Theorem \ref{monot}) by considering the number of column-good matrices with a
fixed row-sum vector, which turns out to be a helpful tool in the proof of our theorems.

In Section \ref{conj}, we discuss a related situation of not-fully occupied arrangements.
for possible directions.




For convenience, we introduce the following notation which shall turn out to be a practical tool.
We will write the row-sums and column-sums of a matrix $M$ in vector form.
Suppose $M$ is a $\pm1$ matrix of $m$ rows and $n$ columns,
but not necessarily in $A(m,n)$.

\begin{defn}\label{defn:rowsum}
(1) Let $\br(M)=(r_1,r_2,r_3,\cdots,r_m)=(1,1,\cdots,1)_{1 \times n} M^T$ denote the \emph{row-sum vector} of 
matrix $M$.
 For a fixed vector $\br$,
 $A_\br(m,n)$ denotes the collection of column-good $\pm1$ matrices (all column-sums are $\pm1$ or $0$) with row-sum vector $\br$,
and let $\alpha_\br(m,n)=|A_\br(m,n)|$.

(2) Similarly,  ${\bf{c}}(M)=(c_1,c_2,c_3,\cdots,c_n)=(1,1,\cdots,1)_{1 \times m} M$ represents the \emph{column-sum vector} of 
matrix $M$.

(3) Let $A_{\br,{\bf{c}}}(m,n)$ be the set of $\pm1$ matrices whose row-sum vector and column-sum vector are $\br$ and ${\bf{c}}$, respectively.
\end{defn}


By definition, it is easy to see  $A(m,n)=\bigcup\limits_{\br:|r_i|\leq 1}A_\br(m,n)$,
so that $\alpha(m,n)=\sum\limits_{\br:|r_i|\leq 1}\alpha_\br(m,n)$.
In particular, if $2 | n$, then $A(m,n)=A_{\mathbf{0}}(m,n)$.

The following fact and its corollary may be compared with Observation \ref{0}.
\begin{observe}\label{observe:rowsum}
Suppose $m$ is even and $n$ is arbitrary.  Let $\br=(r_1,r_2,\cdots,r_{m-1})$ and  ${\bf{r}}^{+}=(r_1,r_2,\cdots,r_{m-1}, -\sum_{i=1}^{m-1} r_i)$
be row-sum vectors of length $m-1$ and $m$, respectively, such that $\bf{r}^{+}$ extends $\br$ and the last entry of $\bf{r}^{+}$ is
equal to the negative sum of all the previous terms. Then there is a natural bijection between $A_\br(m-1,n)$ and $A_{{\bf{r}}^{+}}(m,n)$,
so that $\alpha_\br(m-1,n)=\alpha_{{\bf{r}}^{+}}(m,n)$.
\end{observe}

\begin{proof}
  Straightforward.
\end{proof}

\begin{cor}\label{cor:bij}
Suppose $2|m$. There is an obvious bijection between  $A(m,n)$ and a subset of $A(m-1,n)$, in that
every matrix $M \in A(m,n)$ is mapped to the matrix $M^{-}$ comprised of the top $m-1$ rows of $M$.
As a matter of fact, the image set exactly consists of matrices in $A(m-1,n)$ such that the sum of
all the entries in its row-sum vector is equal to 0 or $\pm1$.
\end{cor}

\begin{proof}
When $2|m$, the sum of every column in $M \in A(m,n)$ must be zero. Of course a matrix $M^{-} \in A(m-1,n)$
may be mapped by deleting the very last row and keeping the top $m-1$ rows of $M$. Conversely, starting from
a matrix $N \in A(m-1,n)$,  a matrix of $m$ rows might be recovered by completing the $j$th column with
an $m$th entry $a_{mj}$ equal to the negative sum of all the entries in the $j$th column of $N$. Nonetheless,
to guarantee $M \in A(m,n)$, the only new condition to meet is that for the newly added row meets the discrepancy requirement,
which is true if and only if $\sum_{j=1}^{n} a_{mj}$ is equal to 0 or $\pm1$.
\end{proof}

Definition \ref{defn:rowsum} will be useful in proving Theorem \ref{porp} and Theorem \ref{appro} in Section 3.

\section{Deciding $\alpha(3,n)$ and $\alpha(4,n)$}\label{34}


Take a matrix $M \in A(3,n)$, consider the upper two rows.
Let $x,y,z,w$ denote the number of 2-tuples
$\binom{+1}{+1},\binom{-1}{-1},\binom{+1}{-1},\binom{-1}{+1}$ in the upper two rows of $M$,
respectively.
Apparently $x+y+z+w=n$ as $M$ has $n$ columns.

\begin{table}[h]
  \centering
\begin{tabular}{|c|c|c||c|c|}
  \hline
   \# of such columns &$x$&$y$&$z$&$w$ \\ \hline
   \multirow{2}{*}{First 2 rows}& +1 & -1 & +1 & -1 \\
   {}&+1 & -1 & -1 & +1 \\
  \hline
\end{tabular}
  \caption{Table 1}  \label{table1}
\end{table}

Let $u, v$ be the number of columns with sum of first two entries zero whose third entry is $1$ or $-1$, respectively.
 Therefore $u+v=z+w=n-x-y$.

\begin{table}[h]
 \centering
\begin{tabular}{|c|c|c||c|c|}
  \hline
   \# of such columns &$x$&$y$&$u$&$v$ \\ \hline
   {Sum of first two rows}&+2 & -2 & 0& 0 \\ \hline
   {3rd row}&-1 & +1 & +1 & -1 \\ \hline
\end{tabular}
  \caption{Table 2}\label{table2}
\end{table}

\begin{lemma}\label{3} For any fixed $M \in A(3,n)$, using the above notation,
If $n$ is even,   $x=y, z=w=u=v$.
If $n$ is odd,  $$x=y\pm1,z=w,u=v; ~or~ x=y\pm1,z=w,u=v\pm2; ~or~ x=y,z=w\pm1, u=v\pm1.$$

It is implied that for any $M \in A(3,n)$,
$$(x-y,z-w,u-v) \in \{(0,1,1),(0,-1,-1),(0,1,-1),(0,-1,1),(1,0,0),(-1,0,0),(1,0,2),(-1,0,-2)\}.$$
\end{lemma}

\begin{proof}
Suppose $n=2k$. As is clear from Table 2, the discrepancy requirement of the first two rows require that $x=y$.
With $x=y$ true, the discrepancy requirement of the third row gives that $u=v$, and Table 1 shows that $z=w$.
Thus $u=v=z=w$ as $u+v=z+w$.

Now suppose $n=2k+1$. If $x=y$ is true, Table 1 shows that $z=w\pm1$ and the last line of Table 2 shows that $u=v\pm1$.
If $x=y$ is false, taken into account of the first two rows of $M$,
Table 2 indicates that $x-y=\pm 1$. Accordingly, either line of Table 1 shows that
$z=w$, and the last line of Table 2 requires that either $u=v$ or $u-v=\pm 2$ (be reminded that $x+y+u+v=2k+1$ is odd).



\end{proof}

\begin{proof}[Proof of Theorem \ref{1}]
For any matrix $M\in A(3,n)$, every column is partially decided by the upper 2 entries. The 4 types
of such columns are listed in Table 1.

In case
the sum of the top two tries is 2 or -2, the bottom entry is uniquely determined as illustrated in
Table 2. There are $x$ and $y$ of them, respectively.


(i) ~Let $n=2k$.
Based on the notation and fact of Lemma \ref{1}, we have $x=y$ and $z=w=u=v$.
First we have to decide on which of the $n$ columns are of types
$\binom{+1}{+1},\binom{-1}{-1},\binom{+1}{-1}$ or $\binom{-1}{+1}$
by looking at the upper two entries.
There are $\binom {2k}{x,x,z,z}$ ways to do this. For the former two types,
the third entry of every column is fully decided as discussed. For the latter two
types, we need to choose $u=z$ columns among the total of $2z$  to make the third
entry $1$ and the rest $v=z$ columns to make the third
entry $-1$.


Hence,
$$\begin{aligned}
\alpha(3,n)&=|A(3,n)|=\sum_{x,y,z,w}\binom n{x,y,z,w}\binom{2z}{z}\\
&=\sum_{x,x,z,z}\binom {2k}{x,x,z,z}\binom{2z}{z}=\sum_{x=0}^{k}\binom {2k}{x,x,k-x,k-x}\binom{2(k-x)}{k-x}\\
&=\sum_{i=0}^{k}\binom{2k}{k-i,k-i,i,i}\binom{2i}{i}=\binom{2k}k\sum_{i=0}^{k}\binom{k}{i}^2\binom{2i}i.
\end{aligned}$$

(ii)~Let $n=2k+1$.
Similar to (i), except that we have more cases to discuss.
By Lemma \ref{1}, we have,
$$\begin{aligned}
&~~~~\alpha(3,n)=|A(3,n)|=\sum_{x,y,z,w,u,v}\binom n{x,y,z,w}\binom{z+w}{u}\\
&=\sum_{x,x+1,z,z}\binom n{x,x+1,z,z}(\binom{2z}{z}+\binom{2z}{z-1})+\sum_{x,x-1,z,z}\binom n{x,x-1,z,z}(\binom{2z}{z}+\binom{2z}{z+1})\\
&+\sum_{x,x,z,z+1}\binom n{x,x,z,z+1}(\binom{2z+1}{z}+\binom{2z+1}{z+1})+\sum_{x,x,z,z-1}\binom n{x,x,z,z-1}(\binom{2z-1}{z}+\binom{2z-1}{z-1})\\
&=2\sum_{z=0}^k\binom n{z,z,k-z,k-z+1}(\binom{2z}{z}+\binom{2z}{z+1})+4\sum_{z=0}^k\binom n{k-z,k-z,z,z+1}\binom{2z+1}{z}\\
&=2\binom{n}{k}\sum_{i=0}^{k}\binom{k}{i}\left[\binom{2i+1}{i+1}\binom{k+1}i+2\binom{k+1}{i+1}\binom{2i+1}{i}\right].
\end{aligned}$$
\end{proof}

\begin{cor}
$$\alpha(4,2k)=\alpha(4,2k-1)=\alpha(3,2k)=\binom{2k}k\sum_{i=0}^{k}\binom{k}{i}^2\binom{2i}i.$$
\end{cor}

\begin{proof}
  By Observation \ref{0} and Theorem \ref{1}.
\end{proof}
Thus we have established Theorem \ref{n=4}.

\begin{cor}\label{4minus3}
$\alpha(4,2k-1)-\alpha(3,2k-1)=-\binom {2k}k \sum_{i=0}^{k-1}\binom ki\binom{k-1}i\binom{2i}{i+1}$.
\end{cor}

\begin{proof}
By comparing the formula of $\alpha(3,2k)$ and that of $\alpha(3,2k-1)$ in Theorem \ref{1}.
\end{proof}

The meaning of Corollary \ref{4minus3} is to attempt to evaluate the difference of $\alpha(2t,2k-1)$ and $\alpha(2t-1,2k-1)$.
We already know that $\alpha(2t,2k-1) \leq \alpha(2t-1,2k-1)$ and that $\alpha(2t,2k)=\alpha(2t-1,2k)$ from Observation \ref{0}.

Calculated by Maple, the initial few entries of the sequence $\{ \alpha(3,2k) \}_{k \geq 0} $ are the following:
$$\{\alpha(3,2k)\}:~1, 6, 90, 1860, 44730, 1172556, 32496156, \cdots.$$

It turns out that $\{ \alpha(3,2k) \}_{n \geq 0} $ is recognized in the database The
On-Line Encyclopedia of Integer Sequences (OEIS) \cite{slo1}
to be sequence $A002896$: Number of $2k$-step polygons on cubic lattice.
This fact is extended by Theorem \ref{thm:bij3}.

\begin{defn} Let $W_n$ be the collection of
the following walks with $n$ steps on the cubic lattice $\mathbb{Z} \times \mathbb{Z} \times
\mathbb{Z}$, starting  at the origin $(0, 0, 0)$,
each step being one of $\{ (1, 0, 0), (-1, 0, 0), (0,1,0), (0,-1,0), (0,0,1), (0,0,-1)\}$.
If $n$ is even, a walk of $W_n$ ends at precisely $(0, 0, 0)$;
if $n$ is odd, a walk of $W_n$ ends either at
one of the six neighboring points of the origin (i.e., $(\pm1,0,0)$ and their permutations)
or one of $\{(1,1,1), (-1,-1,-1)\}$.
\end{defn}

Recall Theorem \ref{thm:bij3} states that there exists a bijection between $A_{3, n}$ and $W_n$ for even $n$.

\begin{proof}[Proof of Theorem \ref{thm:bij3}]
For any given $M \in A_{3, n}$, every column $M_j=(x_j,y_j,z_j)^T$ is a permutation of
the form $(a,a,b)^T$ where $\{a, b\} =
\{-1, 1\}$, $1 \leq j \leq n$.
Let $\varphi ((x_j,y_j,z_j)^T)=\big(\frac{y_j+z_j}{2},\frac{z_j+x_j}{2},\frac{x_j+y_j}{2}\big)^T$.
Reading the columns of $M$ from left to right,
set $\Phi (M)$ to be the walk of $n$ unit steps that starts from the origin using steps $(\varphi((x_j,y_j,z_j)^T), 1 \leq j \leq n)$.

When $n$ is even, the discrepancy requirement gives that
$\sum_{j=1}^{n} \frac{y_j+z_j}{2}=\frac{\sum_{j=1}^{n}y_j}{2}+ \frac{\sum_{j=1}^{n}z_j}{2}=0$.
It is the same for the other two coordinates of the destination point. So the destination is the
origin.
When $n$ is odd, $\Phi (M)$ does not stop at the origin as
$$\sum_{j=1}^{n} \frac{y_j+z_j}{2}+\sum_{j=1}^{n} \frac{z_j+x_j}{2}+\sum_{j=1}^{n} \frac{x_j+y_j}{2}
=\sum_{j=1}^{n}x_j + \sum_{j=1}^{n} y_j + \sum_{j=1}^{n} z_j \in \{-3,-1,1,3 \}.$$
While $\sum_{j=1}^{n} \frac{y_j+z_j}{2}=\frac{\sum_{j=1}^{n}y_j}{2}+ \frac{\sum_{j=1}^{n}z_j}{2} \in \{0,1,-1\}$,
note that $\sum_{j=1}^{n} \frac{y_j+z_j}{2}=1$ if and only if $\sum_{j=1}^{n}y_j=\sum_{j=1}^{n}z_j=1$,
and the consequence is that the destination point is one of \{(1,1,1), (1,0,0)\}.
Hence by symmetry, when $n$ is odd exactly 1 or 3 of the coordinates of $\Phi (M)$
are $\pm 1$, and that the destination is either one of the six neighboring points of the origin or one of
\{(1,1,1), (-1,-1,-1)\}.


Thus $\Phi$ is indeed from $A_{3,n}$ to $W_n$.

Next we check that $\Phi$ is reversible.
Given any unit-step list $W=((a_j,b_j,c_j)^T_{1\leq j\leq n}) \in W_n$,
solving from $\Phi(M)=W$, we get $M_j =(b_j+c_j-a_j,a_j+c_j-b_j,a_j+b_j-c_j)^T$.
Since $W$ consists of only unit steps,
two of $\{a_j, b_j, c_j\}$ are zeros and the third one is 1 or -1,
so that $((b_j+c_j-a_j)+(a_j+c_j-b_j)+(a_j+b_j-c_j))=a_j+b_j+c_j=\pm1$ for any $j$.
This guarantees the column discrepancy condition of $M$.
By definition, $W \in W_n$. If $n$ is even, 
$\sum_{j=1}^{n} a_j =\sum_{j=1}^{n} b_j=\sum_{j=1}^{n} c_j=0$.
so $\sum_{j=1}^{n} b_j+c_j-a_j =\sum_{j=1}^{n} a_j+c_j-b_j=\sum_{j=1}^{n} a_j+b_j-c_j=0$.
If $n$ is odd, the destination of $W$ is either one of the six neighboring points of the origin or one of
\{(1,1,1), (-1,-1,-1)\}. In case $W$ ends at $(1,1,1)$,
the fact $\sum_{j=1}^{n} a_j =\sum_{j=1}^{n} b_j=\sum_{j=1}^{n} c_j=1$
results in $\sum_{j=1}^{n} b_j+c_j-a_j =\sum_{j=1}^{n} a_j+c_j-b_j=\sum_{j=1}^{n} a_j+b_j-c_j=0$.
In case $W$ ends at $(1,0,0)$,  $\sum_{j=1}^{n} a_j =1$ and $\sum_{j=1}^{n} b_j=\sum_{j=1}^{n} c_j=0$
give that \begin{align}
             & \sum_{j=1}^{n} b_j+c_j-a_j =-1, \
             \sum_{j=1}^{n} a_j+c_j-b_j =1, \
             \sum_{j=1}^{n} a_j+b_j-c_j =1. \notag
          \end{align}
Without loss of generality, the other cases are similar.
Hence the row discrepancy condition of $M$ is met too,
and $M$ is indeed in $A_3(n)$.
Thus the bijectivity of $\Phi$ is established.
\end{proof}

According to a result of Richmond and Rousseau
\cite{RR89}, $\alpha(3, 2n) = {2n \choose {n}} \sum_{i=0}^{{n}} {{n} \choose {i}}^2 {{2i \choose i}}$ may be approximated to
$$\frac{(3 \sqrt{3}/4) \cdot {6^{2n}}} {(\pi n)^{3/2}}.$$

By Theorem \ref{1}, with simple calculations on the odd case, we have:
\begin{cor}\label{cor5}For $m=3$,
$\alpha(3,n)\sim\begin{cases}\frac{3\sqrt3}{\sqrt2\pi^{\frac32}}\cdot\frac{6^n}{n^{\frac32}},~&2|n\\
\\ \frac{2\sqrt6}{\pi^{\frac32}}\cdot\frac{6^{n}}{n^{\frac32}},~&2\nmid n\end{cases}.$
\end{cor}

\begin{remark}The 3-dimensional case of Theorem \ref{thm:bij3} can hardly be generalized naturally
to the $n$-dimensional case.  For instance, when $d=2$, there are
in total $\sum_{2i+2j=n} \binom{n}{i,i,j,j}$ ``$n$-step back'' walks
in $\mathbb{Z} \times \mathbb{Z}$, ,
way too many compared with
$\binom{2\lfloor\frac{n+1}{2}\rfloor}{\lfloor\frac{n+1}{2}\rfloor}= \alpha(2, n)$.
\end{remark}

\section{Limits and Bounds}\label{LB}

Let $m$ be an arbitrarily fixed integer.
In this section we investigate
the limits and bounds of the asymptotic behavior of $\alpha(m,n)$ while $n$ tends to infinity.

\subsection{Majority and Comparison while row-sum vector varies}
At the end of Section 1, the notion of row-sum vector is introduced.
For a $\pm1$ matrix $M$ of $m$ rows and $n$ columns not necessarily in $A(m,n)$,
if we modify $M$ by swapping a pair of 1 and $-1$ in the same column,
the column sums do not change, but the row-sum
vector will receive an increment of $(+2,-2,0,0,\cdots,0)$ (or a permutation of the components).
Let $\br \in \mathbb{Z}^m$. Recall that $A_\br(m,n)$ represents the collection of column-good $\pm1$ matrices
(all column-sums are $\pm1$ or $0$) with row-sum vector $\br$,
and $A_{\br,{\bf{c}}}(m,n)$ is the set of $\pm1$ matrices whose row-sum vector and column-sum vector are $\br$ and ${\bf{c}}$, respectively.
We hope to tell whether $|A_\br(m,n)|=\sum\limits_{{\bf{c}}: |c_i|\leq 1} |A_{\br,{\bf{c}}}(m,n)|$
increases or decreases after the swapping.

First we observe that a kind of majorization relation is useful in analyzing the swap process,
which is summarized in the following definition.
\begin{defn}
We say that a vector $x=(x_1,x_2,\cdots,x_m) \in \mathbb{Z}^m$ \emph{majorizes} a vector $y=(y_1,y_2,\cdots,y_m) \in \mathbb{Z}^m$, denoted by $x \succeq  y$,
if (i) $\sum_{i=1}^m x_i=\sum_{i=1}^m y_i$, and (ii) for any $k\in\{1,2,\cdots,m\}$, $\sum_{j=1}^k x_{[i]}\geq\sum_{j=1}^k y_{[i]}$.
Here $(x_{[1]},x_{[2]},\cdots,x_{[m]}), (y_{[1]},y_{[2]},\cdots,y_{[m]})$
are the nondecreasing rearrangements of $(x_1,x_2,\cdots,x_m)$ and $(y_{1},y_{2},\cdots,y_{m})$, respectively.
\end{defn}

The concept and notation of majorization
were introduced in 1952, by Hardy, Littlewood, and P\'olya \cite{HLP52},
but the idea of it seems to first appear in the works of
R. F. Muirhead in 1902 \cite{Muir1902}. Our main result
about majorization is the following Theorem \ref{monot},
which says that for any two ``comparable'' vectors,
in the sense one majorizes the other, there is
a reverse monotonicity in terms of the amount of
column-good $\pm1$ matrices with
corresponding row-sum vectors.


\begin{theorem} \emph{(Decreasing criterion on majorization)} \label{monot}
For any two arbitrary integer row-sum vectors $\br, {\bf{r'}}$ with $\br' \succeq  {\bf{r}}$,
it so holds that $|A_{\bf{r}}(m,n)| \geq |A_{\br'}(m,n)|$.
\end{theorem}

Theorem \ref{monot} is based on Lemmas \ref{lem:Dalton} and \ref{Dtran}, especially
Lemma \ref{lem:monot}.
The key idea is to connect two ``comparable'' vectors $\br$ and ${\bf{r'}}$ with a sequence of
vectors, each of which is 
adjusted slightly from its neighbors.

Let ${\bf{r}} = (r_1, \cdots, r_m)$. We may assume that $r_1 \equiv r_2 \equiv \cdots \equiv r_m \equiv n$ (mod 2),
for otherwise $A_{\bf{r},{\bf{c}}}(m,n)$ would be empty.

\begin{lem}\label{lem:Dalton}
Let ${\bf{r}} = (r_1, \cdots, r_m)$ and ${\bf{r'}} = (r_1', \cdots, r_m')$ be integer vectors with
$\br' \succeq  \br$.
If for any $i$, $r_i \equiv r_i' \equiv n$ (mod 2), then it is always
possible to find a series of finite vectors connecting $\br$ and ${\bf{r'}}$:
$$\br'=\br^{(l)} \succeq \cdots \succeq \br^{(1)} \succeq \br^{(0)}={\bf{r}},$$
such that the difference between any two successive vectors is a permutation of $(+2,-2,0,0,\cdots,0)$.
\end{lem}

Lemma \ref{lem:Dalton} should probably be credited to Muirhead (\cite{Muir1902}, pp.147), and is not very difficult to establish.
The inserted sequence in Lemma \ref{lem:Dalton} is called Dalton's transfer (\cite{IneMaj11}, pp.7).
Observe that for any $k$, $\br^{(k+1)} \succneqq \br^{(k)}$ means that $\br^{(k)}$ is obtained from
reducing 2 from a bigger component of $\br^{(k+1)}$ and adding 2 to a smaller component of $\br^{(k+1)}$.
(We call the above operation an ``adjustment''.)
In case $\br^{(k+1)}$ and $\br^{(k)}$ are both row-sum vectors,
suppose $M^{(k+1)}$ is a matrix in $A_{\br^{(k+1)},{\bf{c}}}(m,n)$.
Equivalent to the above ``adjustment'', a corresponding matrix
$M^{(k)} \in A_{\br^{(k)},{\bf{c}}}(m,n)$ is
obtained by swapping a pair of 1 and $-1$ in a same column of
$M^{(k+1)}$, such that the 1 is taken from a row of $M^{(k+1)}$ with bigger row-sum and
the $-1$ is taken from a row with smaller row-sum. Note that such a pair of
$\pm1$ always exists given that $\br^{(k+1)} \succneqq \br^{(k)}$.

Without loss of generality, the following lemma shows that
if $\br^{(k+1)}$ and $\br^{(k)}$ are
two successive vectors with $\br^{(k+1)} \succneqq \br^{(k)}$
and the only difference of them is
an adjustment at the first two entries,
then it always so holds that
$|A_{\br^{(k)},{\bf{c}}}(m,n)| \geq |A_{\br^{(k+1)},{\bf{c}}}(m,n)|$
for any fixed column-sum vector $\bc$.



\begin{lem}\label{Dtran}
For row-sum vector $\br=(r_1,r_2,r_3,\cdots,r_m)$ with $r_1\geq r_2$, and any fixed column-sum vector ${\bf{c}}$,  we have $|A_{\br,{\bf{c}}}(m,n)|\geq|A_{{\bf{r'}},{\bf{c}}}(m,n)|$, where ${\bf{r'}}=(r_1+2,r_2-2,r_3,\cdots,r_m)$.
\end{lem}

\begin{proof}
We prove it by constructing an injection $\psi$ from $A_{{\bf{r'}},{\bf{c}}}(m,n)$ to $A_{{\bf{r}},{\bf{c}}}(m,n)$.

Take any $\pm1$ matrix $M=(b_{ij})_{1\leq i\leq m}^{1\leq j\leq n}\in A_{{\bf{r'}},{\bf{c}}}(m,n)$.
Since $\sum_{j=1}^n b_{1j}=r_1+2, \sum_{j=1}^n b_{2j}=r_2-2$, we have $\sum_{j=1}^n (b_{1j}-b_{2j}) = r_1-r_2+4 \geq 4$.
Define $s_l=\sum_{j=1}^l(b_{1j}-b_{2j})$. Thus the difference $d_l \triangleq s_l-s_{l-1}=\pm2$ or 0 and
$s_1 \leq 2$. Let $k>0$ be the minimum index such that $s_k=r_1-r_2+2$.
Then $\psi(M)=(a_{ij})_{1\leq i\leq m}^{1\leq j\leq n}$ is simply the matrix obtained by
exchanging the entries of the first two rows of matrix $M$ after the $k$th column, i.e.,
$$a_{ij}=\begin{cases}b_{2j},\quad i=1~and~j>k,\\b_{1j},\quad i=2~and~j>k,\\b_{ij},\quad otherwise.\end{cases}$$
Verification of $\psi(M)\in A_{\br,{\bf{c}}}(m,n)$ is straightforward, and it is easy to see that $\psi$ is an injection.

\end{proof}

Consequently, based on Lemmas \ref{lem:Dalton} and \ref{Dtran}, the following fact  is true.
\begin{lem}\label{lem:monot}
For two arbitrary integer row-sum vectors $\br, {\bf{r'}}$ and
any column-sum vector $\bc$, if $\br' \succeq  {\bf{r}}$,
then $|A_{{\bf{r}}, \bc}(m,n)| \geq |A_{\br',{\bf{c}}}(m,n)|$.
\end{lem}

Now Theorem \ref{monot} is proved as a corollary of Lemma \ref{lem:monot},
as $A_{\br}(m,n)=\bigcup\limits_{\br:|c_i|\leq 1}A_{{\bf{r}}, \bc}(m,n)$
and $A_{\br'}(m,n)=\bigcup\limits_{\br:|c_i|\leq 1}A_{{\bf{r'}}, \bc}(m,n)$.

Note that Theorem \ref{monot} is about comparable row-sum vectors $\br$ and $\bf{r'}$.
What if they are incomparable? We consider the asymptotic behavior and find it even better.

\begin{theorem}\label{lim}
For fixed $m$ and any row-sum vectors $\br=(r_1,r_2,\cdots,r_m), \bf{r'}=(r'_1,r'_2,\cdots,r'_m)$
with $r_1 \equiv \cdots \equiv r_m \equiv r'_1 \equiv \cdots \equiv r'_m$ (mod $2$),
we have $\alpha_\br(m,n) \sim \alpha_{{\bf{r'}}}(m,n)$, i.e.,
$$\lim_{n\rightarrow+\infty}\frac{\alpha_\br'(m,n)}{\alpha_{{\bf{r}}}(m,n)}=1,$$
where as approaching infinity $n$ is required to have the same parity with $r_1$ to ensure that the denominator is not zero.
\end{theorem}

The proof of Theorem \ref{lim} involves an elementary inequality,
which may be shown by induction and for brevity we omit its proof.

\begin{lemma}\label{eineq}
If $a_1,a_2,\cdots,a_k$ are reals, $b_1, b_2, \cdots, b_k, c_1,c_2,\cdots,c_k$ are positive reals, such that $c_1\geq c_2\geq c_3\geq\cdots\geq c_n$ and
$\frac{a_1}{b_1}\geq\frac{a_2}{b_2}\geq\cdots\geq\frac{a_n}{b_n}$.
Then
$$\frac{\sum_{i=1}^k a_i c_i}{\sum_{i=1}^k b_i c_i}\geq\frac{\sum_{i=1}^k a_i}{\sum_{i=1}^k b_i}.$$
\end{lemma}

\begin{proof}[Proof of Theorem \ref{lim}]
Due to Observation \ref{observe:rowsum}, we only discuss the case $m$ is even.
Since $m$ is even, every column-sum is zero, which in turn yields that the sum of all row-sums is zero.
Thus we may assume that $\br$ and $\br'$ are comparable.
Otherwise, we introduce a third row-sum vector $\bf{0}$ with $\br \succeq \bf{0}$ and $\br' \succeq \bf{0}$,
and utilizes the fact  $\lim_{n\rightarrow+\infty}\frac{\alpha_\br'(m,n)}{\alpha_{{\bf{r}}}(m,n)}
=\lim_{n\rightarrow+\infty}\frac{\alpha_\br'(m,n)}{\alpha_{\bf{0}}(m,n)}
\lim_{n\rightarrow+\infty}\frac{\alpha_{\bf{0}}(m,n)}{\alpha_{{\bf{r}}}(m,n)}$.
W.L.O.G., we may further assume $\br' \succeq \bf{r}$ and they are adjacent neighbors as in the statement of Lemma \ref{Dtran}. That is,
$\br=(r_1,r_2,r_3,\cdots,r_m)$ and ${\bf{r'}}=(r_1+2,r_2-2,r_3,\cdots,r_m)$ with $r_1\geq r_2$.
(Of course, again, $r_1,r_2,r_3,\cdots,r_m$ all have the same parity because each of them is the sum of
$n$ $\pm1$'s. )

Theorem \ref{monot} tells us $\frac{\alpha_\br'(m,n)}{\alpha_{{\bf{r}}}(m,n)} \leq 1$.
We are going to show that $\lim_{n\rightarrow+\infty} \frac{\alpha_\br'(m,n)}{\alpha_{{\bf{r}}}(m,n)} \geq 1$.

In order to construct a column-good matrix $M$ in $A_{\br}(m,n)$ (resp. $A_{\br'}(m,n)$) so as to estimate $\alpha_{{\bf{r}}}(m,n)$
(resp. $\alpha_{{\bf{r'}}}(m,n)$), we start from the lower $m-2$ rows. Let the lower $m-2$ rows of $M$ be matrix $M^{-}$.
Note that $M^{-}$ is a $\pm 1$ matrix with row-sum vector $r_3,\cdots,r_m$. While $M^{-}$ does not have to be column-good,
it has to be close to it. Specifically, the column-sum vector of $M^{-}$ must be a permutation of
$$
{\bf{c}_i} =(\underbrace{2,2,\cdots,2}_{i-\frac{r_1+r_2}{2}},
\underbrace{0,0,\cdots,0}_{n+\frac{r_1+r_2}{2}-2i},
\underbrace{-2,\cdots,-2}_{i}),
$$
where $i$ varies appropriately. Each component of $~{\bf{c}}_i$ may be viewed as the sum of the entries in a truncated column of a
matrix $M$ in $A_{\br}(m,n)$ (or, resp., in $A_{\br'}(m,n)$).

Now we construct $M \in A_{\br}(m,n)$. Once $M^-$ is fixed, in order to make the whole column 0-sum,
each pair of entries in the top two rows of $M$ that corresponds to a column with lower entries
summing up to $2$ or $-2$ in ${\bf{c}}_i$ is completely decided.
The other ${n+\frac{r_1+r_2}{2}-2i}$ pairs that correspond to columns with lower entries
summing up to $0$ in ${\bf{c}}_i$, however, have a total of
$\binom{{n+\frac{r_1+r_2}{2}-2i}}{\frac{n+r_1}{2}-i}$ options to decide the relative positions of
$\binom{+1}{-1}$ and $\binom{-1}{+1}$ pairs to meet the top two row-sums $r_1$ and $r_2$.

Below is an illustration of the top two rows of $M$ completed from $M^-$.
\begin{center}
\begin{tabular}{|c|c|c|c|c|}
  \hline
   number of such pairs in $M$ &$i-\frac{r_1+r_2}{2}$&$i$&$\frac{n+r_1}{2}-i$&$\frac{n+r_2}{2}-i$ \\ \hline
   \multirow{2}{*}{pairs in top 2 rows}& $-1$ & $+1$ & $+1$ & $-1$\\
   {}&$-1$ & $+1$ & $-1$ & +1 \\ \hline
   {sum of the bottom $m-2$ entries in the same column}&$+2$ & $-2$ & \multicolumn{2}{|c|}0 \\ \hline
   number of such columns in $M-$ &$i-\frac{r_1+r_2}{2}$&$i$&\multicolumn{2}{|c|}{$n+\frac{r_1+r_2}{2}-2i$}\\
  \hline
\end{tabular}
\end{center}

Now let $h_i=|A_{{\br^{-}},{\bf{c}}_i}(m-2,n)|$, where ${\br^{-}}=(r_3,r_4,\cdots,r_m)$.
To summarize our analysis, we have ${\alpha_{{\bf{r}}}(m,n)}={\sum_{i}\binom{n}{i,i-\frac{r_1+r_2}{2},\frac{n+r_1}{2}-i,\frac{n+r_2}{2}-i}h_{i}}$.
Similarly, ${\alpha_{{\bf{r'}}}(m,n)}={\sum_{i}\binom{n}{i,i-\frac{r_1+r_2}{2},\frac{n+r_1}{2}-i+1,\frac{n+r_2}{2}-i-1}h_{i}}$.

If the roles of rows and columns are interchanged in Lemma \ref{lem:monot}, clearly, $\{h_i \}$ decreases with $i$.
In addition, $$\binom{n}{i,i-\frac{r_1+r_2}{2},\frac{n+r_1}{2}-i+1,\frac{n+r_2}{2}-1-i-1}
/\binom{n}{i,i-\frac{r_1+r_2}{2},\frac{n+r_1}{2}-i,\frac{n+r_2}{2}-i}=\frac{n+r_2-2i}{n+r_1+2-2i}$$
also decreases with $i$.
Hence, by Lemma \ref{eineq} may be omitted from both the numerator and denominator in the following estimations:
$$\begin{aligned}
\frac{\alpha_{{\bf{r'}}}(m,n)}{\alpha_\br(m,n)}
&=\frac{\sum_{i}\binom{n}{i,i-\frac{r_1+r_2}{2},\frac{n+r_1}{2}+1-i,\frac{n+r_2}{2}-1-i}h_{i}}
{\sum_{i}\binom{n}{i,i-\frac{r_1+r_2}{2},\frac{n+r_1}{2}-i,\frac{n+r_2}{2}-i}h_{i}}\\
&\geq\frac{\sum_{i}\binom{n}{i,i-\frac{r_1+r_2}{2},\frac{n+r_1}{2}+1-i,\frac{n+r_2}{2}-1-i}}
{\sum_{i}\binom{n}{i,i-\frac{r_1+r_2}{2},\frac{n+r_1}{2}-i,\frac{n+r_2}{2}-i}}\\
&=\frac{\binom{n}{\frac{n+r_1}{2}+1}\binom{n}{\frac{n+r_2}{2}-1}}{\binom{n}{\frac{n+r_1}{2}}\binom{n}{\frac{n+r_2}{2}}}.
\end{aligned}
$$

The last step above makes use of the Vandermonde Convolution (\cite[page44]{Com74}) in both the top and the bottom.

Finally observe that
$\lim\limits_{n\rightarrow\infty} \frac{\alpha_{{\bf{r'}}}(m,n)}{\alpha_\br(m,n)}
\geq \lim\limits_{n\rightarrow\infty} \frac{\binom{n}{\frac{n+r_1}{2}+1}\binom{n}{\frac{n+r_2}{2}-1}}{\binom{n}{\frac{n+r_1}{2}}\binom{n}{\frac{n+r_2}{2}}}
=\lim\limits_{n\rightarrow\infty}\frac{n-r_1}{n+r_1+2}\cdot\frac{n+r_2}{n-r_2+2}=1$.
\end{proof}

Now we are ready to prove Theorem \ref{porp}, which claims that whenever $m$ is fixed,
$\lim_{n\rightarrow\infty}\frac{\alpha(m,n+2)}{\alpha(m,n)}=\alpha(m,1)^2.$

\begin{proof}[Proof of Theorem \ref{porp}]
For every row-sum vector $\br$, $$A_\br(m,n+2)=\bigcup\limits_{x,y\in A(m,1)}\{(x,y,M)|M\in A_{\br-x-y}(m,n)\},$$
shere $x,y$ are $m\times1$ column vectors.
By Theorem \ref{lim}, for any $\br,\mathbf{x},\mathbf{y}$, $\alpha_{\br-\mathbf{x}-\mathbf{y}}(m,n)$
are asymptotically equivalent.
Thus
$$\alpha_\br(m,n+2)\sim\bigcup\limits_{\mathbf{x},\mathbf{y}\in A(m,1)}|\{(\mathbf{x},\mathbf{y},M)|M\in A_{\br}(m,n)\}|=\alpha(m,1)^2 \alpha_{\br}(m,n)$$
as $n\rightarrow\infty$.
The conclusion follows from
$\alpha(m,n)=\sum\limits_{\br:|r_i|\leq 1}\alpha_\br(m,n)$ with a little careful analysis.
\end{proof}

Recall that by Observation \ref{0}, for any $m$ and any even number $2n$, it always holds that $\alpha(2m,2n)=\alpha(2m-1,2n)$.
\begin{cor}\label{oddeven} For any fixed $m$, letting odd number $2n-1$ approaching infinity, the following limit of ratio is true.
$$\lim_{n\rightarrow\infty}\frac{\alpha(2m,2n-1)}{\alpha(2m-1,2n-1)}=
\frac{|\{(r_1,\cdots,r_{2m}):r_i=\pm1,\sum r_i=0\}|}{|\{(r_1,\cdots,r_{2m-1}):r_i=\pm1\}|}
=\frac{\binom{2m}m}{2^{2m-1}}.$$
\end{cor}

\begin{proof}
Be aware that there is a natural bijection between $A(2m-1,2n-1)$ and $\bigcup_{\br} A_\br(2m,2n-1)$, where
$\br$ is taken over all row-sum vectors $\br=(r_1, \cdots, r_{2m-1}, r_{2m})$ such that
$r_1, \cdots, r_{2m-1} \in \{1, -1\}$ and $r_{2m}=-\sum_{i=1}^{2m-1} r_i$. (This bijection may be viewed as
a generalization of Corollary \ref{cor:bij}.)

On the other hand, $A(2m,2n-1)$ could be written as $A(2m,2n-1)=\bigcup_{\br} A_\br(2m,2n-1)$, where
$\br$ is taken over all row-sum vectors $\br=(r_1, \cdots, r_{2m-1}, r_{2m})$ such that
$r_1, \cdots, r_{2m} \in \{1, -1\}$ and $\sum_{i=1}^{2m-1} r_i=0$.

By Theorem \ref{lim}, the limit of fraction $\lim_{n\rightarrow\infty}\frac{\alpha(2m,2n-1)}{\alpha(2m-1,2n-1)}$
is decided by the ratio of the amounts of row-sum vectors to take for the top and for the bottom.
Therefore the conclusion follows.
\end{proof}

\subsection{Proof of Theorem \ref{appro}}

Recall that Theorem \ref{appro} claims that for fixed $m$, $\alpha(m,n)\sim\Theta\left(\frac{\binom m{m/2}^n}{n^{\frac{m-1}{2}}}\right)$ while
$n$ goes to infinity. In view of Observation \ref{0} and Corollary \ref{oddeven} and Observation \ref{0} hold, in the following
we only consider the case $2 | \gcd (m, n)$. We shall establish the magnitude of $\alpha(m,n)$ by giving lower and upper bounds of constance difference.

\begin{theorem}\label{infn} (lower bound.)
For any fixed even $m$, there exists a constant $C_1(m)$ that depends only on $m$, so that for any even $n$ big enough,
$\alpha(m,n) > C_1(m)\cdot\frac{\binom m{m/2}^n}{n^{\frac{m-1}{2}}}$.
\end{theorem}

To prove Theorem \ref{infn}, we will need the following lemma in probabilistic theory.
\begin{lemma}\label{cltc}
Let $X_1,X_2,X_3,\cdots$ be \emph{i.i.d.} random variables with $Pr(X_i=1)=Pr(X_i=-1)=0.5$ for any $i$.
Then for any $\varepsilon > 0$, $\exists C(\varepsilon)>1$ such that for any $n \in \mathbb{N}_+$,
$Pr(|\sum_{j=1}^n X_j| \geq C(\varepsilon)\sqrt n)<\varepsilon$.
\end{lemma}

Lemma \ref{cltc} is a direct consequence of the central limit theorem (see, for instance, \cite{Fel68}) or \cite{DR10}.

\begin{proof}[Proof of Theorem \ref{infn}]
Recall that $\alpha(m,n)=A_{\mathbf{0}}(m,n)$ as $n$ is even. Let $\Omega=\Omega(m,n)$ denote
the collection of all column-good $\pm1$ matrices.
Set $\cS=\{(r_1,r_2,\cdots,r_m):\sum_{j=1}^m r_j=0; ~2|r_i; ~\forall i\in[m-1], |r_i|<c\sqrt n\}$
to be a collection of row-sum vectors with limited size for every entry except possibly the very last one.
We investigate $D=D_{\cS}=\bigcup_{\br\in \cS} A_\br(m,n)$.
Thus $D$ is a subset of $\Omega$ that contains $A_{\mathbf{0}}(m,n)$.
By Theorem \ref{monot}, $|A_{\mathbf{0}}(m,n)| \geq |A_{\br}(m,n)|$ for any $\br\in \cS$ (noting that
$\mathbf{0}$ and $\br$ are comparable). So,
$$
\alpha(m,n)=|A_{{0}}(m,n)| \geq \frac1{|\cS|}\sum_{\br\in \cS}|A_\br(m,n)|=\frac1{|\cS|}\left|\bigcup_{\br\in \cS}A_\br(m,n)\right|=\frac{|D|}{|\cS|}.
$$

Claim: $|D| > |\Omega|/2$. In fact, we prove that
the probability $P(D) > \frac{1}{2}$ in  the
classical probability space $(\Omega, 2^\Omega, P)$,
where every single matrix $M \in \Omega$
has equal probability $p(M) = \frac1{|\Omega|}=\binom m{m/2}^{-n}$
to show up and as usual $P(A)=\sum_{M \in A} Pr(M \in A)=\sum_{M \in A} p(M) = \frac{|A|}{|\Omega|}$ for every $A \in 2^\Omega$.
This model implies that if a matrix $M$ is to be constructed row by row, and entry by entry, then for each $a_{ij}$,
$Pr(a_{ij}=1)=Pr(a_{ij}=0)=0.5$. The entries $a_{i1}, a_{i2}, \cdots, a_{in}$ on the $i$th row $r_i (M)$ may be viewed
as \emph{i.i.d.} random variables $X_1, X_2, \cdots, X_n$.  Thus for $\varepsilon=\frac{1}{2m}$, applying Lemma \ref{cltc},
there exists a constant $c=C(\frac1{2m})$ depending only on $m$, such that for any $n$,
$Pr(|r_i(M)| \geq c \sqrt n)< \frac{1}{2m}$. Therefore,

$$\begin{aligned}
P(D)&=1-P(\Omega\setminus D)=1-Pr(M \in \Omega\setminus D) \\
&=1-Pr(\exists i\in[m],|r_i(M)|\geq c\sqrt n)\\
&\geq1-\sum_{i=1}^mP(|r_i(M)|\geq c\sqrt n)\\
&>1-\sum_{i=1}^m\frac1{2m}=\frac12.
\end{aligned}$$

Accordingly,
$$\begin{aligned}
\alpha(m,n)&\geq \frac{|D|}{|\cS|}
> \frac{|\Omega|/2}{\left|\{(r_1,r_2,\cdots,r_m):\sum_{i=1}^m r_i=0;  ~2|r_i; ~\forall i\in[m-1], |r_i|<c\sqrt n\}\right|}\\
&\geq\frac{\binom m{m/2}^n /2}{(1+c\sqrt n)^{m-1}}\\
&=\frac1{2 (2c)^{m-1}} \cdot\frac{\binom m{m/2}^n}{n^{\frac{m-1}{2}}} \ (\text{for $n$ big enough}).\end{aligned}$$
Here  $C_1(m)=\frac1{2 (2 C(\frac1{2m}))^{m-1}}$ is a constant determined by $m$.
\end{proof}

Next, in the proof of the upper bound, again it suffices to consider only the case $m$ and $n$ are both even.

\begin{theorem}\label{supn}
(upper bound.)
For any fixed even $m$, there exists a constant $C_2(m)$ that depends only on $m$, so that for any even $n$ big enough,
$\alpha(m,n) \leq C_2(m)\cdot \frac{\binom m{m/2}^n}{n^{\frac{m-1}{2}}}$.
\end{theorem}

To prove Theorem \ref{supn}, we will employ argument similar to what is used in the proof of Theorem \ref{monot}.
First, let $\cR=\{(r_1,r_2,\cdots,r_m):\sum_{j=1}^mr_j=0; ~2|r_i; ~\forall i\in[m-1],|r_i|\leq 1+\sqrt n\}$
be a collection of row-sum vectors with limited size for every entry except possibly the $m$th one.
Thus,
$$|\cR|=\left|\{x:|x|\leq1+\sqrt n,2|x\}\right|^{m-1}\geq(\sqrt n)^{m-1}.$$

Our goal is to find an upper bound for $\alpha(m,n)=|A_{\bf{0}}(m,n)|$.
The following lemma will be needed. Be reminded that for any  $\br \in \cR$, $1 \geq \frac{|A_\br(m,n)|}{|A_{\bf{0}}(m,n)|}$.
\begin{lemma}\label{proportion} Fix even integer $m$.
There exists a positive constant $c=C_3(m)$ depending only on $m$,
such that for all sufficiently large even number $n$ and for any $\br \in \cR$,
$$\frac{|A_\br(m,n)|}{|A_{\bf{0}}(m,n)|} > c >0.$$
\end{lemma}

\begin{proof}
For any fixed $\br=(r_1,r_2,\cdots,r_m) \in \cR$, to construct a column-good matrix $M$ in $A_{\br}(m,n)$ so as to estimate $|A_{\br}(m,n)|$,
we start from the lower $m-2$ rows. Let the lower $m-2$ rows of $M$ be matrix $M^{-}$.
Note that $M^{-}$ is a $\pm 1$ matrix with row-sum vector $r_3,\cdots,r_m$.
While $M^{-}$ does not have to be column-good,
it has to be close to it. Specifically, the column-sum vector of $M^{-}$ must be a permutation of
$$
{\bf{c}_i} =(\underbrace{2,2,\cdots,2}_{n/2-i},
\underbrace{0,0,\cdots,0}_{2i},
\underbrace{-2,\cdots,-2}_{n/2-i}),
$$
where $i$ varies appropriately.
Each component of $~{\bf{c}}_i$ may be viewed as the sum of the entries in a truncated column of a
matrix $M$ in $A_{\br}(m,n)$.

Once $M^-$ is fixed, in order to make the whole column 0-sum,
each pair of entries in the top two rows of $M$ that corresponds to a column with lower entries
summing up to $2$ or $-2$ in ${\bf{c}}_i$ is completely decided.
The other ${n+\frac{r_1+r_2}{2}-2i}$ pairs that correspond to columns with lower entries
summing up to $0$ in ${\bf{c}}_i$ have a total of
$\binom{{n+\frac{r_1+r_2}{2}-2i}}{\frac{n+r_1}{2}-i}$ options to decide the relative positions of
$\binom{+1}{-1}$ and $\binom{-1}{+1}$ pairs.

Now let $h_i=|A_{{\br^{-}},{\bf{c}}_i}(m-2,n)|$, where ${\br^{-}}=(r_3,r_4,\cdots,r_m)$.
We have
\begin{align}\label{alphar}
 {|A_{{\bf{r}}}(m,n)|}={\sum_{i}\binom{n}{i,i-\frac{r_1+r_2}{2},\frac{n+r_1}{2}-i,\frac{n+r_2}{2}-i}h_{i}}
\end{align}

Now we find a lower bound for $|A_{{\bf{r}}}(m,n)|$.
For $\br=(r_1,r_2,\cdots,r_m) \in \cR$,
$s_r=\sum_{r_i>0}r_i=-\sum_{r_i<0}r_i$ is half of the Manhattan distance \cite{RB79} of $\br$ from the origin.
When $n\geq m^2$, $s_r \leq(1+\sqrt n)(m-1) < m \sqrt n$.
Set $s = m \sqrt n$.
By Theorem \ref{monot},$|A_{\pmb{\sigma}}(m,n)|\leq A_\br(m,n)$,
where $\pmb{\sigma}\triangleq(s,-s,0,0,\cdots,0) \succeq (r_1,r_2,\cdots,r_m)$,
Hence,
$$\frac{|A_\br(m,n)|}{|A_{\bf{0}}(m,n)|} \geq \frac{|A_{\pmb{\sigma}}(m,n)|}{|A_{\bf{0}}(m,n)|}
=\frac{\sum_{i}\binom{n}{i,i,\frac {n+s}2-i,\frac {n-s}2-i} h_i}
{\sum_{i} \binom{n}{i,i,\frac n2-i,\frac n2-i}h_{i}}.
$$
(Note that (\ref{alphar}) may be applied to $|A_{\pmb{\sigma}}(m,n)|$ although $\pmb{\sigma}$ is not necessarily in $\cR$.
In the proof of Theorem \ref{monot}, similar argument is used regarding all possible row-sum vectors $\br$,
instead of being limited to $\cR$.)

By a symmetric version of Lemma \ref{lem:monot},  $\{h_i \}$ decreases with $i$.  In addition,
$$\protect\displaystyle \frac{\binom{n}{i,i,\frac {n+s}2-i,\frac {n-s}2-i}}{\binom{n}{i,i,\frac n2-i,\frac n2-i}}
=\frac{1}{\prod_{j=1}^{s} \left( 1+\frac{\frac{s}{2}}{\frac{n}{2}-i-s+j}\right) }
$$
decreases with $i$ as well.  Via Lemma \ref{eineq},
$$\begin{aligned}\frac{A_{\pmb{\sigma}}(m,n)}{A_{\bf{0}}(m,n)}&=
\frac{\sum_{i} {\binom{n}{i,i,\frac {n+s}2-i,\frac {n-s}2-i}} h_i}
{\sum_{i} {\binom{n}{i,i,\frac n2-i,\frac n2-i}}h_{i}}
\geq \frac{\sum_{i} {\binom{n}{i,i,\frac {n+s}2-i,\frac {n-s}2-i}}}
{\sum_{i} {\binom{n}{i,i,\frac n2-i,\frac n2-i}}}\\
&=\frac{\binom{n}{\frac {n+s}2}^2}{\binom{n}{\frac n2}^2} \ (\text{by Vandermonde Convolution})\\
&=\left(\prod_{j=1}^{s/2}\frac{\frac{n-s}2+j}{\frac n2+j}\right)^2
\geq\left(\prod_{j=1}^{s/2}\frac{\frac{n-s}2}{\frac n2}\right)^2=\left(1-\frac sn\right)^{s}
\end{aligned}$$

As $\lim\limits_{n\rightarrow\infty}\left(1-\frac sn\right)^{s}
= \lim\limits_{n\rightarrow\infty} \left(1-\frac{m}{\sqrt n}\right)^{m\sqrt n}=e^{-m^2}$,
it is implied that for sufficiently large even number $n$,
$$\frac{|A_\br(m,n)|}{|A_{\bf{0}}(m,n)|} \geq \frac{|A_{\pmb{\sigma}}(m,n)|}{|A_{\bf{0}}(m,n)|}
>(1-o(1))e^{-m^2}\triangleq C_3(m).$$
\end{proof}

\begin{proof}[Proof of Theorem \ref{supn}]
We still let $\Omega=\Omega(m,n)$ denote
the collection of all column-good $\pm1$ matrices.
By Lemma \ref{proportion},
$$|\Omega| \geq \sum_{\br\in\cR}|A_\br(m,n)|>\sum_{\br\in\cR}C_3(m){|A_{\bf{0}}(m,n)|}=|\cR|\cdot C_3(m){|A_{\bf{0}}(m,n)|}.$$
Consequently,
$$\alpha(m,n)=|A_{\bf{0}}(m,n)|<\frac{|\Omega|}{|\cR|\cdot C_3(m)}\leq \frac{\binom m{m/2}^n}{(\sqrt n)^{m-1}\cdot C_3(m)}=C_2(m)\cdot\frac{\binom m{m/2}^n}{n^{\frac{m-1}{2}}},$$
were $C_2(m)=\frac1{C_3(m)}=(1+o(1))e^{m^2}$.
\end{proof}

Thus, Theorem \ref{appro} is established as a corollary of Theorems \ref{infn} and \ref{supn}.

\section{Discussions}\label{conj}


A more general case of  combinatorial rectangles with desired discrepancy is as follows.
Some but not necessarily all the squares in an $m$ by $n$ grid is filled with 1 or $-1$.
The other squares are left empty, or filled with 0 if we like to look at them that way.
The requirement is still that the absolute value of the sum of every row and of every column
must not exceed 1.
Here we are interested with the ``dense'' case. That is, assume that the total number of occupied squares is $c mn$,
where $c \in (0,1]$ is a fixed constant.  We study the situation
when $mn\rightarrow\infty$.

Obviously, a ``good'' arrangement is still ``good'' after exchanging some of its rows and
respectively, some of columns. Thus we may reorder the rows and columns so that the number of
filled squares (i.e. squares occupied by 1 or $-1$) are in decreasing order.
W.L.O.G., assume that the number of filled rows (i.e. rows with at least one filled square)
is at least as many as the number of filled columns.
Or to say, we investigate $\alpha(n_1, n_2, \cdots, n_l)$,
where $n_i$ represents the number of filled squares in row $i$
with $\sum_{i=1}^{l} n_i = cmn$ and $n_1 \geq n_2 \geq \cdots n_l$.
Thus $\alpha(n_1) = \binom{n_1}{\frac{n_1}{2}}$ or $\binom{n_1+1}{\frac{n_1 +1}{2}}$ depending
on the parity of $n_1$ (or as in Observation \ref{observe:1}
simply $\binom{2\lfloor\frac{n_1+1}{2}\rfloor}{\lfloor\frac{n_1+1}{2}\rfloor}$).
But what next? Unlike the fact $\alpha(2,n)=\alpha(1,n)$ as seen in Observation \ref{observe:1},
the formula is complicated even for the 2-row case ($l=2$).
$$
\alpha(n_1, n_2)=  \sum_{i=0}^{\min(n_2, n-n_1)} \binom{n_1-n_2+2i}{i}
\sum_{\br=(r_1, -r_1): |r_1| \leq n_2-i} \alpha_{\br}(2, n_2-i) \alpha_{{\mathbf{-r_1}}}(1, n_1-n_2+i) \alpha_{{\bf{r_1}}}(1,i).
$$
The above summation index $i$ represents the number of filled squares of row 2 that are not on the columns of
filled squares of row 1. Our rearrangement assumption forces the overlapped columns of filled squares of
the two rows must be column 1 through column $n_2-i$.
But the other $i+(n_1-n_2+i)$ columns have options to get ordered.  It would be interesting to investigate further.



\bibliographystyle{plain}
\bibliography{sy}

\end{document}